\documentclass[12pt]{article}  
\usepackage{graphicx} 
\usepackage{color}

\usepackage[english]{babel}
\usepackage[utf8]{inputenc}
\usepackage[T1]{fontenc}
\usepackage{colortbl} 
\usepackage[a4paper,left=2.5cm,right=2.5cm, top=2.25cm ,bottom=2.25cm, twoside]{geometry}
\usepackage{xcolor}
\usepackage{stmaryrd}
\usepackage{amssymb}
\usepackage{amsmath}
\usepackage{libertine}
\usepackage{arydshln}
\usepackage[nonumberlist]{glossaries}
\usepackage{calc}
\usepackage[]{titlesec} 
\definecolor{linkColor}{HTML}{32a852}
\usepackage[colorlinks=true,citecolor=linkColor,linkcolor=black]{hyperref}
\usepackage{fancyhdr}
\usepackage{lipsum}

\usepackage{setspace}
\usepackage[subfigure]{tocloft}
\usepackage{subfigure}
\usepackage{caption}

\usepackage{amssymb}
\usepackage{amsmath}
\usepackage{graphicx}
\usepackage{latexsym}
\usepackage{wrapfig}
\usepackage{placeins}
\usepackage{lineno}
\usepackage{float}
\usepackage{cleveref}
\usepackage{hyperref}

\usepackage{tikz,array,graphicx,amsfonts,amsmath,amssymb,amsthm,mathrsfs}
\newtheorem{theorem}{Theorem}[section] 

\newtheorem{lemma}[theorem]{Lemma}

\usepackage{wrapfig}
\usepackage[nottoc]{tocbibind}
\usepackage{graphicx}
\usepackage{placeins}
\usepackage{float}

\usepackage{comment}

\renewcommand{\div}{{\hbox{div}}}
\renewcommand{\a}{\frac{2}{\delta x^2}}

\title{Flux-equilibrated based \,\\a posteriori error analysis for an interface problem with CutFEM}

\author{Daniela Capatina and Aimene Gouasmi}
\date{ }
\begin{document}


\maketitle

\begin{abstract}
This paper addresses the local recovery of conservative fluxes and the a posteriori error analysis for an elliptic interface problem with discontinuous coefficients. The transmission conditions on the interface are imposed by means of Nitsche's method and the discretization is carried out using conforming finite elements on unfitted meshes via the CutFEM method. A flux is subsequently defined in the global Raviart-Thomas space, ensuring that it satisfies the natural conservation property on the cut cells, and is then employed in the a posteriori error analysis. We prove here the sharp reliability of the error estimator and show a numerical experiment which illustrates the approach.
\end{abstract}



\section{Introduction}

The importance of reconstructing equilibrated fluxes from primal discrete solutions is widely recognized in computational mathematics \cite{vohralik2011guaranteed, Dana2016, Aimene}. One application is in a posteriori error analysis \textcolor{black}{and adaptive mesh refinement} \cite{Aposteriori1, Aposteriori2}, where the difference between the numerical and the recovered flux serves as a reliable error estimator. 

This paper investigates a 2D second-order elliptic interface problem with discontinuous coefficients and standard transmission conditions across the interface. The numerical approximation of the solution is achieved through the Cut Finite Element Method  \cite{CutFEM2}, which is designed for cases where the mesh does not fit the interface. Our goal is to recover locally conservative fluxes in the Raviart-Thomas space, by using the methodology introduced in \cite{Dana2016} for the Poisson equation on fitted meshes. The reconstruction relies on an auxiliary mixed problem, whose primal solution coincides with the original finite element solution, whereas the multiplier defined on the mesh edges is used in the definition of the flux. Notably, these multipliers are computed locally by solving explicit linear systems at each vertex, in contrast to other reconstruction techniques such as \cite{vohralik2011guaranteed} where solving a mixed problem is necessary.   

In this work, we focus on the interface diffusion problem, discretized by means of piecewise linear conforming elements on an unfitted mesh. First, we describe how to construct an 
$ H(div,\Omega)$-conforming flux in the whole domain, using a hybrid mixed formulation with Lagrange multipliers associated to each sub-domain. This approach yields a flux which satisfies the natural local conservation property while preserving the transmission condition on the interface. We then employ the recovered flux in a posteriori error analysis and establish the sharp reliability of the corresponding error estimator, \textcolor{black}{i.e. we bound the error by the estimator plus a higher order term, with a constant in front of the estimator equal to $1$.} Finally, we present a numerical experiment which illustrates the proposed approach.


\section{The continuous and discrete problems}
\label{sec:1}

Let $\Omega$ be a 2D polygonal domain with an interface $\Gamma$ dividing $\Omega$ into two non-overlapping sub-domains: $\bar{\Omega} = \bar{\Omega}^1 \cup \bar{\Omega}^2$, where $\partial \Omega^1 \cap \partial \Omega^2 = \Gamma$. The unit normal vector to $\Gamma$, denoted $n_{\Gamma}$, is oriented from $\Omega^1$ to $\Omega^2$. We consider the following model problem:
\begin{equation}\label{eq: continuous_problem}
	\left\{
	\begin{aligned}
		-\text{div}(K\nabla u^{i}) = f^i \quad & \text{in } \Omega^i, \,\, i = 1, 2, \\
		u = 0 \quad & \text{on } \partial \Omega,\\
		[u] = 0, \;\; [K \nabla u \cdot n_{\Gamma}] = 0 \quad & \text{on } \Gamma,
	\end{aligned}
	\right.
\end{equation}
where $[u] = u^1 - u^2$ represents the jump across $\Gamma$. Assume $f^i \in L^2(\Omega^i)$ and for simplicity, let $K|_{\Omega^i} = k_i > 0$ for $i = 1, 2$. The study can be extended to the case of piecewise constant positive definite tensors $K$ and to a non-zero jump of the normal fluxes across the interface.

For the finite element approximation of \eqref{eq: continuous_problem}, we introduce the following notation. Let $\mathcal{T}_h$ denote a triangular regular mesh of $\Omega$, whose elements are closed sets, and $\mathcal{F}_h$ the set of edges. The diameter of $T \in \mathcal{T}_h$ (and {the length} of $F \in \mathcal{F}_h$) is denoted $h_T$ (and $h_F$). For an interior {edge} $F$, $n_F$ is a fixed unit normal vector to $F$, oriented from $T_F^-$ to $T_F^+$, where $T_F^-$ and $T_F^+$ are the two triangles sharing $F$. If $F \subset \partial \Omega$, then $n_F$ is the outward normal vector to $\Omega$ whereas if $F \subset \Gamma$, then $n_F = n_{\Gamma}$. For $\omega \subset \mathbb{R}^d$ with $1 \leq d \leq 2$, denote the $L^2(\omega)$-norm by $\|\cdot\|_\omega$ and the $L^2(\omega)$-orthogonal projection onto $P^{{m}}(\omega)$ by $\pi^{{m}}_{\omega}$, for ${m} \in \mathbb{N}$. For $i = 1, 2$, we define:
\begin{equation*}
	\mathcal{T}_h^{i} = \big\{ T \in \mathcal{T}_h \ ; \ T \cap \Omega^{i} \neq \emptyset \big\}, \quad \mathcal{F}_h^{i} = \big\{ F \in \mathcal{F}_h \ ; \ F \cap \Omega^{i} \neq \emptyset \big\}
\end{equation*}
and we set $\Omega_h^i = \bigcup_{T \in \mathcal{T}_h^i} T$; note that $\Omega^i \subset \Omega_h^i$. For cut elements, let:
\begin{equation*}
	\begin{split}
		\mathcal{T}_h^{\Gamma} = \big\{ T \in \mathcal{T}_h \ ; \ T \cap \Gamma \neq \emptyset \big\}, & \quad \mathcal{F}_h^{\Gamma} = \big\{ F \in \mathcal{F}_h \ ; \ F \cap \Gamma \neq \emptyset \big\}, \\
		\mathcal{F}_g^i = \big\{ F \in \mathcal{F}_h^i \ ; \ (T_F^+ \cup T_F^-) \cap \Gamma \neq \emptyset \big\}, & \quad T^i = T \cap \Omega^i \,\,\, \forall T \in \mathcal{T}_h^{\Gamma}.
	\end{split}
\end{equation*}

In order to focus on the flux reconstruction, we assume here that $\Gamma$ is a polygonal line such that for each $T \in \mathcal{T}_h^{\Gamma}$, the intersection $\Gamma_T = T \cap \Gamma$ is a line. For a function $v$ discontinuous across $\Gamma$, let $v^i = v_{|\Omega^i}$ and define the two following means at a point $x \in \Gamma$:
\begin{equation*}
	\{v\}(x) = \omega_1 v^1(x) + \omega_2 v^2(x), \quad \{v\}^*(x) = \omega_2 v^1(x) + \omega_1 v^2(x),
\end{equation*}
where the weights $\omega_1$, $\omega_2$ and the harmonic mean $k_{\Gamma}$ are given (cf. \cite{Ern}) by:
\begin{equation*}
	\omega_1 = \frac{k_2}{k_2 + k_1}, \quad \omega_2 = \frac{k_1}{k_1 + k_2}, \quad k_{\Gamma} = \frac{k_1 k_2}{k_1 + k_2}.
\end{equation*}

Additionally, we introduce the arithmetic mean $\langle v \rangle = \frac{1}{2}(v^- + v^+)$ and the jump $[\![v]\!] = v^- - v^+$ across an interior edge $F \in \mathcal{F}_h^i$; for a boundary edges, we set $\langle v \rangle = [\![v]\!] = v$. Finally, for $i = 1, 2$, let
$V^i = \big\{ v \in H^1(\Omega_h^i) \ ; \ v_{|(\partial\Omega^i \setminus \Gamma)} = 0 \big\}$ and
\begin{equation*}
	\mathcal C_h^i = \big\{ v \in V^i \ ; \ v_{|T} \in P^1(T), \,\, \forall T \in \mathcal{T}_h^{i} \big\}, \quad \mathcal C_h = \mathcal C_h^{1} \times \mathcal C_h^{2}.
\end{equation*}

We consider the well-posed weak formulation of problem \eqref{eq: continuous_problem}: Find $u \in H^1_0(\Omega)$ such that
\begin{equation}\label{eq: continuous_problem_weak}
	\int_{\Omega} K \nabla u \cdot \nabla v \, dx = \int_{\Omega} f v \, dx \quad \forall v \in H^1_0(\Omega).
\end{equation}

For its numerical approximation, we apply Nitsche's method  
to enforce the transmission conditions on $\Gamma$ as in \cite{hansbo2002unfitted}, and use the CutFEM approach \cite{CutFEM2} to ensure the robustness with respect to the interface geometry by adding a ghost penalty stabilization term. The discrete problem reads: Find $u_h = (u_h^1, u_h^2) \in \mathcal C_h$ such that
\begin{equation}\label{eq:Primal_conform_Formulation}
	a_h(u_h, v_h) = l_h(v_h) \quad \forall v_h \in \mathcal C_h,
\end{equation}
where:
\begin{equation*}
	\begin{split}
 a_h(u_h, v_h) = & \sum_{i=1}^{2} \bigg({\sum_{T \in \mathcal T_h^{i}}} \int_{{T \cap} \Omega^i} k_i \nabla u_h^i \cdot \nabla v_h^i \, dx +  \sum_{F \in \mathcal{F}_g^i} \gamma_g h_F \int_{F} k_i [\![\partial_n u_h^{i}]\!] [\![\partial_n v_h^{i}]\!] \, ds\bigg)\\
		 &+ \sum_{T \in \mathcal T_h^{\Gamma}} \int_{\Gamma_T} \bigg( \frac{\gamma k_{\Gamma}}{h_T} [u_h] [v_h] - \{K \nabla u_h \cdot n_{\Gamma}\} [v_h] - \{K \nabla v_h \cdot n_{\Gamma}\} [u_h] \bigg) \, ds,\\
	l_h(v_h) = & \sum_{i=1}^{2} \int_{\Omega^i} f^i v_h^i \, dx.
	\end{split}
\end{equation*}

The stabilization parameters $ \gamma_g, \,\gamma>0 $ are chosen independently of the mesh/interface geometry and diffusion coefficients. For any $ v_h\in \mathcal{C}_h $, we define the norm:
\begin{equation*}
	\|v_h\|_h^2 = \sum_{i=1}^2  \bigg( k_i||\nabla v_h^i||^2_{\Omega^i} +  \sum_{F \in \mathcal{F}_g^i} k_i  h_F  || [\![\partial_n v_h^{i}]\!] ||^2_F\bigg) + \sum_{T \in \mathcal{T}_h^\Gamma}  \frac{k_\Gamma}{h_T}||  [v_h]|| ^2_{\Gamma_T}.
\end{equation*}

For sufficiently large $ \gamma $, it is well-known that  $ a_h(\cdot,\cdot) $ is uniformly coercive with respect to $ \|\cdot\|_h $ on $ \mathcal{C}_h $. Consequently, problem \eqref{eq:Primal_conform_Formulation} is well-posed thanks to the Lax-Milgram theorem.


\section{The auxiliary mixed formulation}
\label{sec:2}

Let $\mathcal D_h=\mathcal D_h^1\times \mathcal D_h^2$ and $\mathcal M_h=\mathcal M_h^1\times \mathcal M_h^2$, where for $i=1,2$, we define:
\[
\begin{split}
	\mathcal D_h^i &= \{v \in L^2(\mathcal T_h^i) ;\, v|_T \in P^1(T), \, \forall T \in \mathcal T_h^i \}, \\
	\mathcal M_h^i &= \{\mu \in L^2(\mathcal F_h^i) ;\, \mu|_F \in P^1(F), \, \forall F \in \mathcal F_h^i, \,  
	\textstyle{\sum_{F \in \mathcal{F}_N} \mathfrak{s}_{F,N} h_F \mu|_F(N)} = 0, \, \forall N \in \overset{\circ}{\mathcal {N}^i_h}\}.
\end{split}  
\]
Here above, $\overset{\circ}{\mathcal{N}_h^i}$ denotes the set of nodes interior to $\Omega_h^i$, $\mathcal{F}_N$ is the set of edges connected to the node \(N\) and \(\mathfrak{s}_{F,N}$ is a sign function which takes the value \(1\) or \(-1\), depending on the orientation of \(n_F\) relative to the clockwise rotation sense around \(N\). These spaces are equipped with the following norms: for $v_h \in \mathcal D_h$ and $\mu_h \in \mathcal M_h$,
\[
\|{v_h}\|_{\mathcal D_h}^2 = \|v_h\|_h^2 + \sum_{i=1}^2 \sum_{F \in \mathcal F_h^i}  k_i h_F^{-1}||[\![v_h^i]\!]||_F^2, \qquad 
	\|{\mu_h}\|_{\mathcal M_h}^2 = \sum_{i=1}^2 \sum_{F \in \mathcal F_h^i}  k_i h_F ||\mu_h^i||_F^2.
\]

Following the ideas of \cite{Dana2016,Aimene}, we introduce the following auxiliary mixed formulation: Find \((\tilde{u}_h, \theta_h) \in \mathcal D_h \times \mathcal M_h\) such that

\begin{equation} \label{eq: Mixed formulation}
	\begin{split}
		\tilde{a}_h(\tilde{u}_h, v_h) + b_h(\theta_h, v_h) &= l_h(v_h), \quad \forall v_h \in \mathcal{D}_h, \\
		b_h(\mu_h, \tilde{u}_h) &= 0, \qquad\;\; \forall \mu_h \in \mathcal{M}_h,
	\end{split}
\end{equation}
where $\tilde{a}_h(\cdot, \cdot) = a_h(\cdot, \cdot) - d_h(\cdot, \cdot)$ and $b_h(\mu_h, v_h) = \sum_{i=1}^2 b_h^i(\mu_h^i, v_h^i)$, with:
\[
\begin{split}
	d_h(\tilde{u}_h, v_h) &= \sum_{i=1}^2 \sum_{F \in \mathcal{F}_h^i} \int_{F \cap \Omega^i} 
	\left(\langle k_i \nabla \tilde{u}_h^i \cdot n_F \rangle [\![v_h^i]\!] + \langle k_i \nabla v_h^i \cdot n_F \rangle [\![\tilde{u}_h^i]\!] \right) \, \mathrm{d}s, \\
	b_h^i(\mu_h^i, v_h^i) &= \sum_{F \in \mathcal{F}_h^i} \frac{k_i h_F}{2} \sum_{N \in \mathcal{N}_F} \mu_h^i|_F(N)[\![v_h^i]\!](N) \approx \sum_{F \in \mathcal{F}_h^i} \int_F k_i \mu_h^i [\![v_h^i]\!] \, \mathrm{d}s.
\end{split}
\]
Here above, $\mathcal{N}_F$ is the set of nodes belonging to \(F\). We have obtained similar results to those in \cite{Dana2016,Aimene}; we give them below and we refer the reader to  \cite{these} for further details.

\begin{lemma} \label{thrm: kernel_b_h}
The space $\mathrm{Ker}\, b_h = \{v_h \in \mathcal{D}_h ;\, b_h(\mu_h, v_h) = 0, \, \forall \mu_h \in \mathcal{M}_h \} $ is equal to \(\mathcal C_h\).
\end{lemma}

\begin{theorem} \label{thrm: inf-sup_cond}
There exists a constant \(\beta>0\) independent of \(h\), \(K\) and \(\Gamma\) such that:
\[
	\inf_{\mu_h \in \mathcal{M}_h} \sup_{v_h \in \mathcal{D}_h} \frac{b_h(\mu_h, v_h)}{\|{\mu_h}\|_{\mathcal M_h} \|{v_h}\|_{\mathcal D_h}} \geq \beta.
\]
\end{theorem}

Since $d_h(\cdot,\cdot)$ vanishes on $\mathcal C_h\times \mathcal C_h$, Lemma \ref{thrm: kernel_b_h} ensures the coercivity of \(\tilde{a}_h(\cdot, \cdot)\) on  \(\mathrm{Ker}\, b_h\), as well as the equivalence between the primal and mixed formulations \eqref{eq:Primal_conform_Formulation} and \eqref{eq: Mixed formulation}. Consequently, \(u_h = \tilde{u}_h\). Then the Babuska-Brezzi theorem yields, thanks to Theorem \ref{thrm: inf-sup_cond}, the well-posedness of \eqref{eq: Mixed formulation}. Note that the multipliers \(\theta_h^i\) and the bilinear forms \(b_h^i(\cdot, \cdot)\) are defined over the entire edges, such that on a cut edge \(F \in \mathcal{F}_h^\Gamma\) there are two multipliers living on \(F\). 

An important aspect of this approach is that \(\theta_h^i\) can be computed locally, as sum of local contributions defined on patches associated with the nodes. For more details, we refer to \cite{these}.


\section{Recovery of a conservative flux in the  global Raviart-Thomas space}\label{sec: sigma_Hdiv_global}
We  present here the local reconstruction of a flux $\sigma_h$ in the global Raviart-Thomas space $\mathcal{ RT}^m_h(\Omega)\subset H(\div,\Omega)$, for $m=0$ or $m=1$.    
The definition is given by imposing the degrees of freedom of $ \sigma_h $. The multipliers $\theta_h^i$ are used as corrections of the normal derivative of the solution, in the definition of the normal flux. We differentiate the cut elements from the non-cut ones. Thus, for  any edge $F\in \mathcal F_h$ and any function $ v \in P^m(F) $, we  impose:
\begin{equation}\label{flux-definition-edge-specific}
		\displaystyle \int_{F}\sigma_h\cdot n_Fv\,ds =
		\begin{cases}
			\displaystyle 	\int_{F}\langle {k_i\nabla u_h^i\cdot n_F}\rangle v\,ds - b_F^i(\theta^i_h,v), &  \text{ if } F\in \mathcal F_h^i\backslash \mathcal F_h^\Gamma,\\[3mm]
			\displaystyle{\sum_{i=1}^{2}\bigg(\int_{F\cap\Omega^i}\langle {k_i\nabla u_h^i\cdot n_F}\rangle v\,ds -b_F^i(\theta^i_h,v)\bigg), } & \text{ if } F\in \mathcal F_h^{\Gamma},\, F \not \subset \Gamma,\\[3mm]
			\displaystyle{ \int_{F}\{K\nabla u_h\cdot n\} v\,ds 
				-\gamma k_\Gamma h_F^{-1}\int_{F} [{u_h}]v\,ds,} & \text{ if } F\subset\Gamma.
		\end{cases}
\end{equation}
These relations allow to uniquely define $\sigma_h\cdot n_F$ in $P^m(F)$. In addition, for $m=1$, we also define interior degrees of freedom as follows: for any $T\in \mathcal T_h$ and any $\zeta\in (P^0(T))^2$, we set
\begin{equation}\label{flux-definition-interior}
		\begin{split}
			\int_{T}\sigma_h\cdot \zeta\,dx
			&= \sum_{i=1}^{2}\int_{T\cap \Omega^i}k_i\nabla u_h^i\cdot \zeta\,dx
			-\int_{\Gamma_T} \{K \zeta\cdot n_\Gamma\}[{u_h}]\,ds,  \\
			&+\sum_{i=1}^{2}\sum_{F\in\mathcal F_g\cap\mathcal{F}_h^{i} \cap \partial T}
			\gamma_g k_ih_F\int_{F}[\![{\nabla u_h^{i}\cdot n_{F}}]\!] [\![{\zeta\cdot n_{F}}]\!]\,ds.
		\end{split}
\end{equation}
	
Note that $\sigma_h$ strongly satisfies the transmission condition on $\Gamma$, that is \([\sigma_h\cdot n_\Gamma]=0\). Moreover, the flux satisfies the natural local conservation property as stated in the next theorem.   
\begin{theorem}
Let $f\in L^2(\Omega)$ defined by $f_{|\Omega^i}=f^i$, for $i=1,2$. One has that: 
\begin{equation}\label{eq:conservation_equalities}
(\div \,\sigma_h)_{|T}=-\pi^m_T f, \qquad \forall T\in \mathcal{T}_h.
\end{equation}
\end{theorem}
\noindent The proof, based on the mixed formulation and on the definition of $\sigma_h$, can be found in \cite{these}.


\section{Application to a posteriori error analysis}\label{sec: A poster_sigma_Hdiv_glob}

We set  $\tau_h=K^{-1/2}(\sigma_h-K\nabla_h u_h)$, where the discrete gradient is defined by  restriction on any $T\in \mathcal T_h$: $(\nabla_h u_h)_{|T\cap \Omega^i}= (\nabla u_h^ i)_{|T\cap \Omega^i}$ for $i=1,2$. We define the local error estimators:
\begin{equation*}
\eta_{T} =\|\tau_h\|_{T},\quad \forall T\in \mathcal T_h;\qquad
\tilde{\eta}_T^2=  \frac{ h_T k_\Gamma}{|\Gamma_T|h_T^{min}}\| [u_h]\|_{\Gamma_T}^2,\quad \forall T\in \mathcal T_h^\Gamma
\end{equation*}
where  $h_T^{min}$ is the shortest part of the cut edges of $T$: $h_T^{min}=\min\{|F^i|; \  F\in \partial T\cap \mathcal F_h^\Gamma, \  1\leq i \leq 2 \}$.

Let $ \delta_T= k_i $ if $ T\in \mathcal T_h^i\backslash\mathcal T_h^\Gamma $ and  $ \delta_T= k_\Gamma $ if $T\in \mathcal T_h^\Gamma $. The corresponding global error estimators and the data approximation are given by: 
\begin{equation*}
	\eta^2 =\sum_{T\in \mathcal{T}_h}\eta_{T}^2 =\|\tau_h\|_{\Omega}^2,\qquad 
	\quad \eta_{\Gamma}^2=
	\sum_{T\in\mathcal T_h^\Gamma}\tilde{\eta}_T^2,\quad \epsilon(\Omega)^2= 
	\sum_{T\in \mathcal T_h}\frac{h_T^2}{\delta_T}\|f-\pi^m_T f\|_{T}^2,
\end{equation*}

\subsection{Reliability}
Thanks to the fact that $\sigma_h$  strongly satisfies the transmission condition across the interface, we are able to establish the reliability of the a posteriori error estimator $\eta+ \eta_\Gamma$. Let us put  $|v|_{1,K,h} :=\|K^{1/2}\nabla_h v\|_{\Omega}$, for any $v\in H^1(\Omega^1)\times H^1(\Omega^2)$.
	
\begin{theorem}\label{thm: first_bound_rb}
There exists a constant $C$ independent of the mesh, the coefficients and the interface geometry such that  
\begin{equation}
		| u-u_h)|_{1,K,h}\leq \eta +  \inf_{v\in H_0^1(\Omega)}| v- u_h|_{1,K,h} +C\epsilon(\Omega).
\end{equation}
\end{theorem}
\begin{proof}
Let $\sigma=K\nabla u$ and let  $\varphi \in H^1_0(\Omega)$ be the unique solution of the weak problem:
\begin{equation}\label{eq:aposteriori_phi}
		\int_{\Omega}K\nabla\varphi\cdot \nabla w \,dx = \sum_{i=1}^2\int_{\Omega^i}K\nabla u_h^i\cdot \nabla w \,dx \quad \forall w \in H^1_0(\Omega).
\end{equation}
	
The triangle inequality yields that: 
\begin{equation}\label{eq: triangl_inq_2}
		|u-u_h|_{1,K,h}\leq  |u-\varphi|_{1,K,h}+ |\varphi-u_h|_{1,K,h},
\end{equation}
where, thanks to the definition of $\tau_h$, we have:
\begin{equation*}
|u-\varphi|_{1,K,h}^2=\int_{\Omega}\nabla(u-\varphi)\cdot(\sigma-\sigma_h)\,dx +\int_{\Omega} K^{1/2}\nabla(u-\varphi)\cdot\tau_h\,dx+\int_{\Omega} \nabla(u-\varphi)\cdot K(\nabla_hu_h-\nabla \varphi)\,dx.
\end{equation*}
The last term vanishes thanks to (\ref{eq:aposteriori_phi}) with the test-function $w:=u-\varphi$, so we further get that:
\begin{equation}\label{eq: aposteriori_step2_2}
			|u-\varphi|_{1,K,h}^2\leq \eta|u-\varphi|_{1,K,h}+ \left |\int_{\Omega}\nabla(u-\varphi)\cdot(\sigma-\sigma_h)\,dx  \right|.
	\end{equation}
We integrate by parts the last term of \eqref{eq: aposteriori_step2_2}. Since  $(u-\varphi)\in H^1_0(\Omega)$ and $(\sigma-\sigma_h)\in H(\div,\Omega)$, we obtain using the conservation property (\ref{eq:conservation_equalities}) that:
\begin{equation*}
\int_{\Omega}\nabla(u-\varphi)\cdot(\sigma-\sigma_h)\,dx =\int_{\Omega}(f-\pi_T^mf)(u-\varphi)\,dx.
\end{equation*}
Thanks to the Cauchy-Schwarz inequality, to the fact  that $k_\Gamma \leq k_i$  $(1\le i\le 2)$ and to the orthogonal projection's property ($\|w-\pi_T^m w\|_T\lesssim h_T \|\nabla w\|_T$ with $w=u-\varphi$), one gets:
\begin{equation*}
\begin{split}
\left|	\int_{\Omega}(f-\pi_T^mf)(u-\varphi)\,dx\right| &\lesssim |u-\varphi|_{1,K,h}\bigg( \sum_{i=1}^2\sum_{T\in\mathcal T_h^i\backslash \mathcal T_h^\Gamma}\frac{h_T^2}{k_i}\|f-f_h\|_{T}^2+\sum_{T\in \mathcal T_h^\Gamma}\frac{h_T^2}{k_{\Gamma}}\|f-f_h\|_{T}^2\bigg)^{1/2}\\
&=|u-\varphi|_{1,K,h} \epsilon(\Omega).
\end{split}
\end{equation*}
Using the previous estimate in \eqref{eq: aposteriori_step2_2}, we end up with the following bound:
\begin{equation}\label{eq: aposteriori_bound_1}
|u-\varphi|_{1,K,h}\leq \eta + C\epsilon (\Omega).
\end{equation}

On the other hand, from problem (\ref{eq:aposteriori_phi}) we have that $|\varphi-u_h|_{1,K,h}= \displaystyle{\inf_{v\in H_0^1(\Omega)}}|v-u_h|_{1,K,h}$. Using the last relation and (\ref{eq: aposteriori_bound_1}) in \eqref{eq: triangl_inq_2} yields the result.
\end{proof}

Since  
\begin{equation}\label{eq:inf_bound}
\inf_{v\in H_0^1(\Omega)}|v-u_h|_{1,K,h}\leq |I_hu_h-u_h|_{1,K,h},	
\end{equation}
where $I_hu_h$ is any approximation of $u_h$ belonging to $H^1_0(\Omega)$, the next step consists in constructing  an interpolation $I_h u_h\in H^1_0(\Omega)$ such that the interpolation error is bounded by $\eta_\Gamma$. In what follows, we focus on the cut triangles since on $T\in\mathcal T_h\backslash \mathcal T_h^\Gamma$, we simply take $I_hu_h=u_h$.

Let $T\in \mathcal T_h^\Gamma$. Without loss of generality, we denote the triangle $T$ by $\triangle A_1A_2A_3$ and let $M\in A_1A_2$,  $N\in A_1A_3$ denote the intersection points  with $\Gamma$.  For simplicity  of notation, we denote by  $T^\triangle$ and $T^\square$ the triangular and quadrilateral parts of the cut triangle $T$, respectively, and assume that $T^\triangle=T^1$, $T^\square= T^2$. Since the ratios   $|A_2M| /|A_3N|$ and  $|A_3N| / |A_2M|$ cannot  blow up simultaneously, we assume, without loss of generality, that $|A_2M| / |A_3N|$ is bounded and then we cut $T^\square$ into two triangles using $A_3M$, see Figure \ref{Fig: cut_I_h_TRiangle} (otherwise, we use $A_2N$).
\begin{figure}[ht]
\centering
\includegraphics[width=0.28\textwidth]{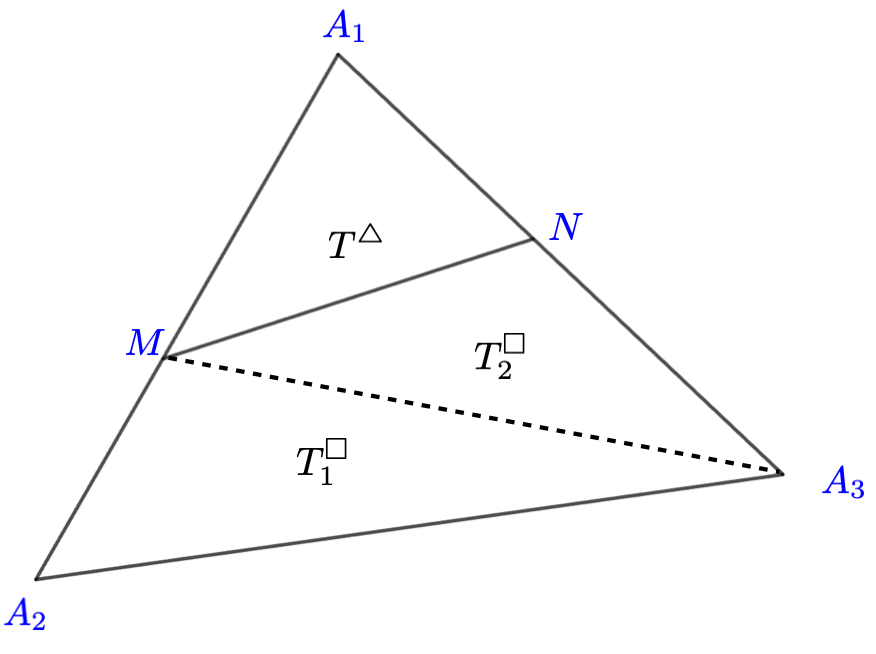}
\caption{Notation on a cut element for the interpolation}
\label{Fig: cut_I_h_TRiangle}
\end{figure}

We now define the interpolation operator $I_h$ as follows: $I_hu_h$ is a linear function on each sub-triangle $\triangle A_1MN=:T^\triangle$, $\triangle MNA_3=:T^\square_2$ and $\triangle MA_2A_3=:T^\square_1$, and satisfies: 
\begin{equation}\label{eq: def_I_h}
I_hu_h(A_i)=u_h(A_i) \quad (1\le i \le 3),\quad I_hu_h(M)=\{u_h\}^*(M),\quad I_hu_h(N)=\{u_h\}^*(N) .
\end{equation}

This choice guarantees that $I_h u_h\in H^1(T)$, since it is piecewise linear on $T$ and continuous at the points $M$, $N$ and $A_3$. Clearly, $I_hu_h$ belongs to $H^1(T)$ for any $T\in \mathcal T_h$ and $[\![{I_hu_h}]\!]_F=0$ for any edge $F\in \mathcal F_h$, hence $I_hu_h\in H^1_0(\Omega)$.    

\begin{theorem}\label{thm:interp_bound} One has that $|I_hu_h-u_h|_{1,K,h}\lesssim  \eta_\Gamma$. 
\end{theorem} 
 
\begin{proof} Clearly, $|I_hu_h-u_h|_{1,K,h}^2=\sum_{T\in \mathcal T_h^{\Gamma}} |I_hu_h-u_h|_{1,K,T}^2$ so let $T\in \mathcal T_h^{\Gamma}$. Since 
\begin{equation}\label{eq: decomposition_termGamma}
|I_h u_h-u_h|_{1,K,T}^2=|I_hu_h-u_h^1|_{1,K,T^\triangle}^2+|I_hu_h-u_h^2|_{1,K, T^\square_1}^2+ |I_hu_h-u_h^2|_{1,K,T^\square_2}^2,
\end{equation}
we have to bound  $(I_hu_h-u_h)_{|\tilde T}$ for any $\tilde T\in \{T^\triangle,T^\square_1,T^\square_2\}$. For $\tilde T\in \Omega^i$, one has thanks to \eqref{eq: def_I_h}:
\begin{equation*}
(I_hu_h-u_h)_{|\tilde T}= (-1)^i\omega_i[u_h](M) \varphi_M +(-1)^i\omega_i[u_h](N)\varphi_N, 
\end{equation*}
where  $\varphi_M, \varphi_N\in P^1(\tilde T)$ are the nodal basis functions associated to $M$ and $N$, respectively. 

So in order to bound the three terms of \eqref{eq: decomposition_termGamma}, we first express $\varphi_M$ and $\varphi_N$ on each of the three sub-triangles, in terms of the barycentric coordinates $(\lambda_j)_{1\leq j\leq 3}$ of $T$. After some technical computations (detailed in \cite{these}), one gets for $\tilde T=T^\triangle$ that
\begin{equation}\label{eq: phi_M_T_delta}
(\varphi_M)_{|T^\triangle}=\frac{|F_3|}{|A_1M|}\lambda_2, \quad (\varphi_N)_{|T^\triangle}=\frac{|F_2|}{|A_1N|}\lambda_3,
\end{equation}
where $|F_2|=|A_1A_3|$ and $|F_3|=|A_1A_2|$. Meanwhile, for $\tilde T= T^\square_1$ and  $\tilde T=T^\square_2$, one obtains that: 
\begin{equation}\label{eq: phi_M_T_squar_1}
(\varphi_M)_{|T^\square_1}=\frac{|F_3|}{|A_2M|}\lambda_1,	\quad (\varphi_N)_{|T^\square_1}=0,
\end{equation}
\begin{equation}\label{eq: phi_M_T_squar_2}
(\varphi_M)_{|T^\square_2}=\frac{|F_3|}{|A_1M|}\lambda_2,	 \quad (\varphi_N)_{|T^\square_2}= \frac{|F_2|}{|NA_3|} \lambda_1 - \frac{|F_2|}{|NA_3|} \frac{|MA_2|}{|MA_1|}\lambda_2.
\end{equation}

We next bound $|I_hu_h-u_h|_{1,K,\tilde{T}}^2$ on each sub-triangle $\tilde  T$. Let $\alpha =[u_h](M)$ and $\beta=[u_h](N)$.

Let $\tilde  T=T^\triangle$ first. Using that  $k_\Gamma=k_i\omega^i$ and that  $0\leq \omega^i\leq 1$, one gets that
\begin{equation*}
|I_hu_h-u_h^1|_{1,T^\triangle,K}^2 \leq  k_\Gamma |W_h^\triangle|_{1,T^\triangle}^2
\end{equation*}
where $W_h^\triangle\in P^1(T^\triangle)$ satisfies
$W_h^\triangle(M)= -\alpha$, $W_h^\triangle(N)=-\beta$ and $W_h^\triangle(A_1)=0$. Thus, it follows that $W_h^\triangle=-\alpha \varphi_M-\beta \varphi_N$. From \eqref{eq: phi_M_T_delta} and using also that 
\begin{equation*}
|\lambda_i|_{1,T^\triangle}^2=\frac{|T^\triangle|}{|T|}|\lambda_i|_{1,T}^2\quad (1\leq i\leq 3),\qquad  |\lambda_i|_{1,T}\lesssim 1,\qquad |T|\approx h_T^2,
\end{equation*}
one  obtains the following bound:  
\begin{equation}\label{eq: W_h_0}
|W_h^\triangle|_{1,T^\triangle}^2\lesssim \frac{|T^\triangle|}{|A_1M|^2}\alpha^2+\frac{|T^\triangle|}{|A_1N|^2}\beta^2. 
\end{equation}

Similarly, from \eqref{eq: phi_M_T_squar_1} and \eqref{eq: phi_M_T_squar_2} and the properties of $\lambda_j$, we deduce that:
\begin{equation}\label{eq: W_h_1}
|W_h^1|_{1,T^\square_1}^2\leq \alpha^2 \frac{|T^\square_1|}{|T|} \frac{|F_3|^2}{|A_2M|^2} |\lambda_1|_{1,T}^2\lesssim \frac{|T^\square_1|}{|A_2M|^2}\alpha^2,
\end{equation}
as well as 
\begin{equation}\label{eq: W_h_2}
|W_h^2|_{1,T^\square_2}^2\lesssim \alpha^2 \frac{|T^\square_2|}{|A_1M|^2} + \beta^2 \frac{|T^\square_2|}{|A_3N|^2}\frac{|F_3|^2}{|A_1M|^2}.
\end{equation} 

Combining \eqref{eq: W_h_0}, \eqref{eq: W_h_1} and \eqref{eq: W_h_2}, one gets that:
\begin{equation}\label{eq: bound interpolation E_A_E_b}
|I_hu_h-u_h|_{1,K,T}^2\lesssim k_\Gamma \alpha^2\underset{E_\alpha(T)}{\underbrace{\bigg( \frac{|T^\triangle|}{|A_1M|^2}+ \frac{|T^\square_1|}{|A_2M|^2}+\frac{|T^\square_2|}{|A_1M|^2}\bigg)}} + k_\Gamma \beta^2\underset{E_\beta(T)}{\underbrace{\bigg(  \frac{|T^\triangle|}{|A_1N|^2}+ \frac{|T^\square_2|}{|A_3N|^2}\frac{|F_3|^2}{|A_1M|^2} \bigg)}}.
\end{equation} 

The next step consists in bounding the expressions $E_\alpha(T)$ and $E_\beta(T)$, in the best possible way with respect to the geometry of the cut cell.
Developing the expressions of $ E_\alpha(T) $ and  $ E_\beta(T)$ and using that $|A_1N|+|A_3N|=|F_2|$ and $|A_1M|+|A_2M|=|F_3|$, one obtains cf. \cite{these} that:  
\begin{equation}\label{eq: E_alpha_}
E_\alpha(T)=\frac{|T|}{|A_1M||A_2M|},\quad E_\beta(T)= \frac{|T|}{|F_2|}\bigg(\frac{|A_1M|}{|A_1N||F_3|}+\frac{|F_3|}{|A_1M||A_3N|}\bigg).
\end{equation}
Using that $|A_1M|\leq |F_3|$, we further obtain that: 
\begin{equation}\label{eq: E_beta_}
E_\beta(T)\leq \frac{|T|}{|F_2|}\bigg(\frac{1}{|A_1N|}+\frac{|F_3|}{|A_1M||A_3N|}\bigg)=\frac{|T|}{|A_1N||A_3N|}+\frac{|T|}{|F_2||A_1M|}\frac{|A_2M|}{|A_3N|}.
\end{equation} 

Recalling that $h_T^{min}= \min\{|A_1M|,|A_2M|, |A_1N|,|A_3N|\}$ and that $|A_1M|+|A_2M|=|F_3|$, we deduce from  \eqref{eq: E_alpha_} and \eqref{eq: E_beta_}, after some straightforward calculations, that:
\begin{equation*}
E_{\alpha}(T)\leq \frac{2|T|}{h_T^{min}|F_3|}\lesssim \frac{h_T}{h_T^{min}}, \quad E_\beta \lesssim  \frac{h_T}{h_T^{min}}+ \frac{h_T}{h_T^{min}} \frac{|A_2M|}{|A_3N|}.
\end{equation*}
Our initial choice of cutting $T^\square$ into two sub-triangles ensures that $\displaystyle\frac{|A_2M|}{|A_3N|}$ is bounded (otherwise, we change the roles of the points $ M $ and $N$). So we finally get
\begin{equation}\label{eq: E_final}
E_\alpha\lesssim\frac{h_T}{h_T^{min}}, \quad 
E_\beta \lesssim  \frac{h_T}{h_T^{min}}.
\end{equation} 

Using \eqref{eq: E_final} in \eqref{eq: bound interpolation E_A_E_b} yields that:
\begin{equation*}
|I_hu_h-u_h|_{1,K,T}^2\lesssim \frac{h_Tk_\Gamma}{h_T^{min}} (\alpha^2+\beta^2)\lesssim \frac{h_Tk_\Gamma}{h_T^{min}|\Gamma_T|}\|[u_h]\|_{\Gamma_T}^2=\tilde{\eta}_T^2,\quad \forall T\in \mathcal T_h^{\Gamma},
\end{equation*}
hence summing upon the cut triangles ends the proof.
\end{proof}

Thanks to \eqref{eq:inf_bound} and to Theorems \ref{thm: first_bound_rb} and \ref{thm:interp_bound}, we finally obtain the desired reliability bound.
\begin{theorem}[Reliability]
There exists $C>0$ independent of the mesh, the coefficients and the interface geometry such that  
\begin{equation}
	|u-u_h|_{1,K,h}\leq \eta +C( \eta_\Gamma+ \epsilon(\Omega) ).
\end{equation}
\end{theorem}

\subsection{Numerical experiment}\label{sec: Numerical_test_RT1_global}
    
We consider an interface problem characterized by an intricate petal-shaped interface, see \cite{he2019residual}. The exact solution is defined by means of the level set function $\phi$ as follows:
\[
u(x, y) = 
\begin{cases} 
\frac{1}{k^-} \phi(x, y), & \text{if } \phi(x, y) < 0, \\
\frac{1}{k^+} \phi(x, y), & \text{if } \phi(x, y) \geq 0,
\end{cases}, \quad \phi(x, y) = (x^2 + y^2)^2 \left(1 + 0.5 \sin\left(12 \tan^{-1}\left(\frac{y}{x}\right)\right)\right) - 0.3,
\]
where \((x, y) \in \Omega = [-1, 1]^2\) and $ k^-=1, \, k^+=100k^- $. We set a non-homogeneous  Dirichlet condition. \textcolor{black}{As regards the refinement, we use the D\"{o}rfler marking strategy \cite{dorfler1996convergent}, i.e. we look for the set of elements \(\mathcal T_h^m\) with minimal cardinal such that  \( \theta \eta(\mathcal T_h)^2 \leq \eta(\mathcal T_h^m)^2\). In the adaptive mesh refinement procedure, the marking percent \(\theta\) is set to be 20\%, i.e. the ordered elements that account to the top 20\% of the total error estimator get refined.} As stopping criteria, we  impose that the total number of degrees of freedom $N$ is less than $20 000$.

\begin{figure}[h]
	\centering
	{\includegraphics[width=0.3\textwidth]{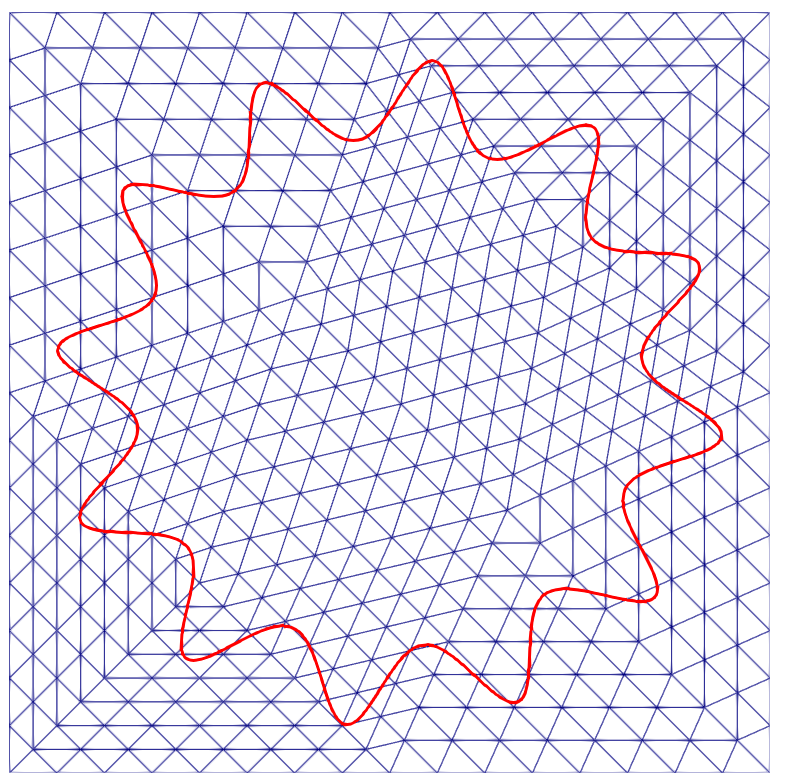}}
	{\includegraphics[width=0.3\textwidth]{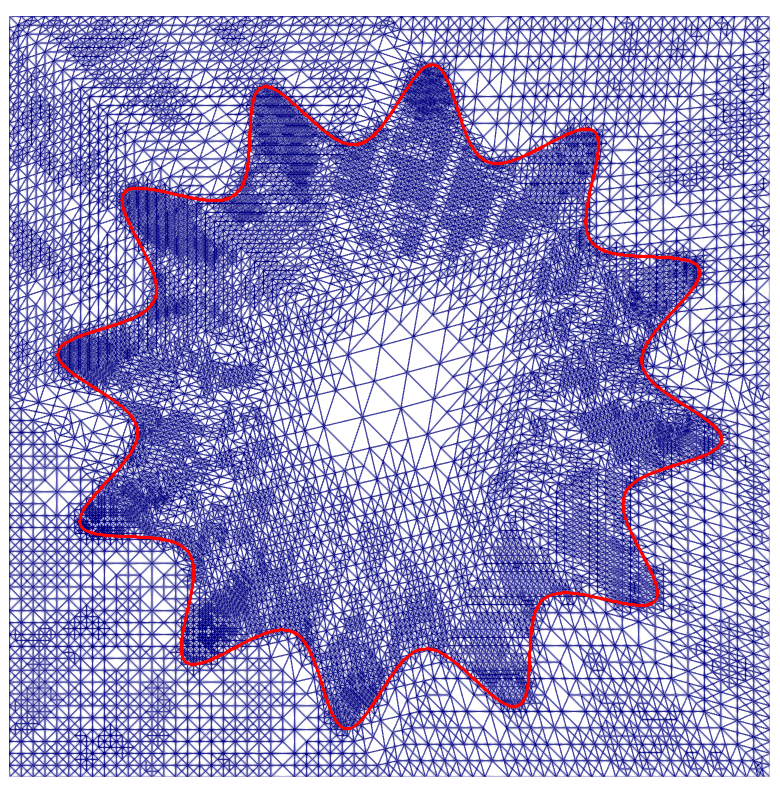}}
	\includegraphics[width=.36\textwidth]{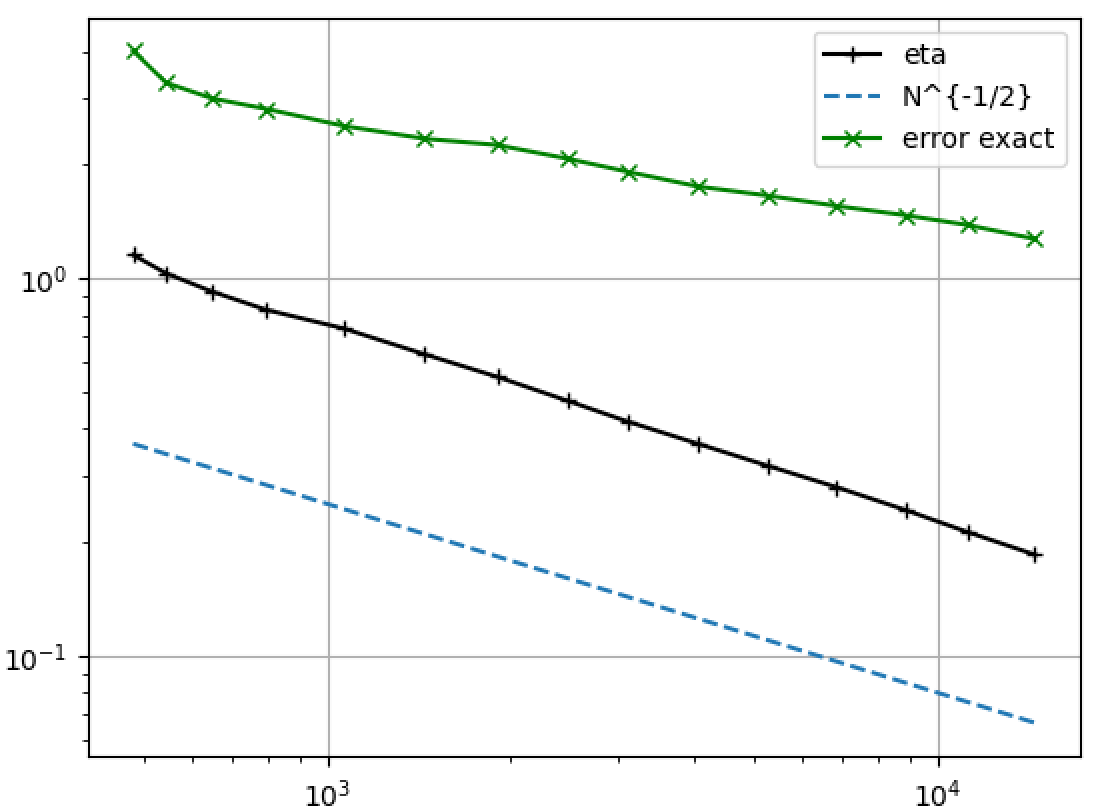}
	\caption{Initial and adaptive final meshes. Error slope}
	\label{fig1:Meshes_Ptal}
\end{figure}

Figure \ref{fig1:Meshes_Ptal} shows the initial mesh (left) and the final mesh after 14 iterations (middle), as well as the obtained convergence rates (right). We observe that the refinement around the interface is rather dense,  due to the large curvature of the interface. The convergence plot of Figure \ref{fig1:Meshes_Ptal}  indicates the optimal rate decay for both the error and estimator, that is $ O(N^{-1/2}) $.

\vspace{1cm}

\textbf{ This project has received funding from the European Union's Horizon H2020 Research and Innovation under the Marie Curie Grant Agreement $\text{N}^{\circ}$ 945416.}
\bibliographystyle{plain}
\bibliography{bibmodel_2}


\hspace{(-0.6cm)}Aimene Gouasmi \\
LMAP \& CNRS UMR 5142, Universit\'e de Pau et des Pays de l'Adour, IPRA BP 1155, 64013 Pau /  LAMPS,
Universit\'e de Perpignan Via Domitia, B\^atiment B3, Avenue Paul Alduy, 66860 Perpignan \\
\texttt{aimene.gouasmi@univ-pau.fr}

\end{document}